\documentclass[12pt]{amsart}

\usepackage[margin=1.2in]{geometry}
\usepackage{amssymb, amsmath, amsthm, mathrsfs, longtable, wrapfig, multicol, wasysym, subcaption, booktabs, cancel, comment}
\usepackage{asymptote, graphicx}
\usepackage{latexsym}
\usepackage{tkz-euclide}
\usepackage{hyperref}
\usepackage{enumerate}
\usepackage{epigraph}

\usepackage[final]{pdfpages}

\newtheorem{theoremIntro}{Theorem}[]

\newtheorem{theorem}{Theorem}[section]
\newtheorem{lemma}[theorem]{Lemma}
\newtheorem{proposition}[theorem]{Proposition}

\newtheorem{definition}[theorem]{Definition}
\newtheorem{corollary}[theorem]{Corollary}

\theoremstyle{remark}

\numberwithin{equation}{section}

\newcommand{\calE}{\ensuremath{\mathcal{E}}}
\newcommand{\calF}{\ensuremath{\mathcal{F}}}
\newcommand{\calG}{\ensuremath{\mathcal{G}}}

\newcommand{\calI}{\ensuremath{\mathcal I}}
\newcommand{\calJ}{\ensuremath{\mathcal J}}

\newcommand{\calP}{\ensuremath{\mathcal{P}}}
\newcommand{\calQ}{\ensuremath{\mathcal{Q}}}

\newcommand{\calR}{\ensuremath{\mathcal{R}}}
\newcommand{\calS}{\ensuremath{\mathcal{S}}}
\newcommand{\calT}{\ensuremath{\mathcal{T}}}

\newcommand{\calV}{\ensuremath{\mathcal{V}}}
\newcommand{\calW}{\ensuremath{\mathcal {W}}}
\newcommand{\calX}{\ensuremath{\mathcal {X}}}

\newcommand{\tors}{\ensuremath{\mathrm{tors}}}

\newcommand{\bfone}{\ensuremath{\mathbf{1}}}

\newcommand{\bbN}{\ensuremath{\mathbb{N}}}
\newcommand{\bbZ}{\ensuremath{\mathbb{Z}}}

\begin{document}

\title[Infinite co-minimal pairs involving lacunary sequences]{Infinite co-minimal pairs involving lacunary sequences and generalisations to higher dimensions}

\author{Arindam Biswas}
\address{Department of Mathematics, Technion - Israel Institute of Technology, Haifa 32000, Israel}
\curraddr{}
\email{biswas@campus.technion.ac.il}
\thanks{}

\author{Jyoti Prakash Saha}
\address{Department of Mathematics, Indian Institute of Science Education and Research Bhopal, Bhopal Bypass Road, Bhauri, Bhopal 462066, Madhya Pradesh,
India}
\curraddr{}
\email{jpsaha@iiserb.ac.in}
\thanks{}



\begin{abstract}
	The study of minimal complements in a group or a semigroup was initiated by Nathanson. The notion of minimal complements and being a minimal complement leads to the notion of co-minimal pairs which was considered in a prior work of the authors. In this article, we study which type of subsets in the integers and free abelian groups of higher rank can be a part of a co-minimal pair. We show that a majority of lacunary sequences have this property. From the conditions established, one can show that any infinite subset of any finitely generated abelian group has uncountably many subsets which is a part of a co-minimal pair. Further, the uncountable collection of sets can be chosen so that they satisfy certain algebraic properties.
\end{abstract}

\maketitle

\section{Introduction and Motivation}

 Let $(G,+)$ be an abelian group and $W\subseteq G$ be a nonempty subset. A nonempty set $W'\subseteq G$ is said to be an \textit{additive complement} to $W$ if $W + W' = G.$ Additive complements have been studied since a long time in the context of representations of the integers e.g., they appear in the works of Erd\H{o}s, Hanani, Lorentz and others. See \cite{Lorentz54, Erdos54, ErdosSomeUnsolved57} etc. The notion of minimal additive complements for subsets of groups was introduced by Nathanson in \cite{NathansonAddNT4}. An additive complement $W'$ to $W$ is said to be minimal if no proper subset of $W'$ is an additive complement to $W$, i.e., 
$$W + W' = G \,\text{ and }\, W + (W'\setminus \lbrace w'\rbrace)\subsetneq G \,\,\, \forall w'\in W'.$$ 
Minimal complements are intimately connected with the existence of minimal nets in groups. See \cite[Section 2]{NathansonAddNT4} and \cite[Section 2.1]{MinComp1}. Further, in case of the additive group $\mathbb{Z}$, they are related to the study of minimal representations. See \cite[Section 3]{NathansonAddNT4}.

Given two nonempty subsets $A, B$ of a group $G$, they are said to form a co-minimal pair if $A \cdot B = G$, and $A' \cdot B \subsetneq G$ for any $\emptyset \neq A' \subsetneq A$ and $A\cdot B' \subsetneq G$ for any $\emptyset \neq B' \subsetneq B$. Thus, they are pairs $(A\in G,B\in G)$ such that each element in a pair is a minimal complement to the other. Co-minimal pairs in essence capture the tightness property of a set and its complement. Further, they are a strict strengthening of the notion of minimal complements in the sense that a non-empty set $A$ might admit a minimal complement, but might not be a part of a co-minimal pair. See \cite[Lemma 2.1]{CoMin1}. 

Which sort of sets can be a part of a co-minimal pair is an interesting question. It was shown in \cite[Theorem B]{CoMin1} that if $G$ is a free abelian group (of any rank $\geq 1$), then, given any non-empty finite set $A$, there exists another set $B$ such that $(A,B)$ forms a co-minimal pair. 
Moreover, in a very recent work of Alon, Kravitz and Larson, they establish that any nonempty finite subset of an infinite abelian group is a minimal complement to some subset \cite[Theorem 2]{AlonKravitzLarson}, and this implies that the statement of \cite[Theorem B]{CoMin1} holds for nonempty finite subsets of any infinite abelian group. 
However, the existence of non-trivial co-minimal pairs involving infinite subsets $A$ and $B$ were unknown, until recently. It was shown in \cite{CoMin2} that such pairs exist and explicit constructions of two such pairs in the integers $\mathbb{Z}$ were provided. The aim of this article is to establish that infinite co-minimal pairs are abundant along majority of lacunary sequences in the integers and from there, draw conclusions on which sort of infinite subsets can be a part of a co-minimal pair. We are considering the underlying set associated with the sequence and this implies that there is some sparseness between successive elements of the  underlying set. These subsets also satisfy a number of algebraic and combinatorial properties. Moreover, they can be generalised to any free abelian group of finite rank.

\subsection{Statement of results}

First, let us recall the notion of lacunary sequences.

\begin{definition}[Lacunary sequence]
	A Lacunary sequence is a sequence of numbers $\lbrace x_{n}\rbrace_{n\in \mathbb{N}}$ such that $\frac{x_{n+1}}{x_{n}}\geqslant \lambda >1 \forall n\in \mathbb{N}$.
\end{definition}
Additive complements of lacunary sequences also has a long history. For a brief history see \cite[Section 1]{RuzsaLacunary}. However, the study of minimal additive complements and co-minimal pairs involving lacunary sequences is new.  Our first result concerns co-minimal pairs of lacunary sequences. To avoid introducing cumbersome notation from the beginning, we state a simplified version of Theorem \ref{Thm} below 

\begin{theoremIntro}
\label{ThmA}
     In the additive group of the integers, a ``majority'' of lacunary sequences have the property that they belong to a co-minimal pair. 
\end{theoremIntro}

By a ``majority'', we mean the following: we know that a lacunary sequence $\lbrace x_{n}\rbrace_{n\in \mathbb{N}}$ is defined by the growth condition $\frac{x_{n+1}}{x_{n}}\geqslant \lambda$, with $\lambda \in (1,+\infty)$. The lacunary sequences which satisfy Theorem \ref{ThmA} have $\lambda \in[6,+\infty)$ and in some cases even $\lambda \in (3,+\infty)$. Further, they can be generalised to $\mathbb{Z}^{d}$. See Theorem \ref{Thm} for the complete statement. It is worth mentioning that in \cite{CoMin2}, it has been established that the set consisting of the terms of the lacunary sequence $1, 2, 2^2, 2^3, \cdots$ is a part of a co-minimal pair.  

Next, we consider the subsets of $\bbZ$ of the following types. 
\begin{enumerate}[(Type 1)]
\item Symmetric subsets of $\bbZ$ containing the origin,
\item Symmetric subsets of $\bbZ$ not containing the origin,
\item Subsets of $\bbZ$ which are bounded from the above or from the below. 
\end{enumerate}
We investigate whether a subset of $\bbZ$ of the above types can be a part of a co-minimal pair with a subset of $\bbZ$ of the above types. We establish certain sufficient conditions for this to be true. In fact, we have the following result.

\begin{theoremIntro}
For (I, II) equal to any one of 
$$(1, 1), (1, 2), (1, 3), (2, 1), (2, 2), (2, 3), (3, 3),$$
there are uncountably many subsets of $\bbZ$ of Type I which form a co-minimal pair with a subset of $\bbZ$ of Type II.

\end{theoremIntro}

Again, the above is a specialised version of a more general theorem which holds for $\mathbb{Z}^{d}$. See Theorem \ref{Thm:Uncountable}.

\section{Co-minimal pairs involving lacunary sequences and generalisations}

Let $d\geq 1$ be an integer. Denote the lexicographic order on $\bbZ^d$ by the symbol $<$. The element $(1, \cdots, 1)$ of $\bbZ^d$ is denoted by $\bfone$.
For $?\in \{<, \leq, > , \geq\}$ and $x\in \bbZ^d$, the set $\{n\in \bbZ^d\,|\, n?x\}$ is denoted by $\bbZ^d_{?x}$. 
Let $t_0 < t_1 < t_2 < \cdots$ be elements of $\bbZ^d$. Assume that $t_0 > 0$, 
$$t_n \geq 2 t_{n-1}
\quad
\text{ for all } 
n \geq 1,$$
and for some $n\geq 0$, all the coordinates of $t_n$ are positive. Let $\calT, \calV, \calW$ be subsets of $\bbZ^d$ defined by 
\begin{align*}
\calT 
& = \{t_0 , t_1, t_2, \cdots \},\\
\calV 
& = \calT \cup (-\calT),\\
\calW 
& = \calT \cup \{0\} \cup (-\calT).
\end{align*}

Then, the following holds,
\begin{theorem}
\label{Thm}
\quad
\begin{enumerate}
\item 
If
$$t_n > 3 t_{n-1}
\quad
\text{ for all } 
n \geq 1,$$
then the set 
$\calW$
and a symmetric subset $\calE$ of $\bbZ^d$ containing the origin form a co-minimal pair in $\bbZ^d$. 
\item 
If
$$t_n > 3 t_{n-1}
\quad
\text{ for all } 
n \geq 1,$$
then the set 
$\calW$
and a symmetric subset $\calF$ of $\bbZ^d$ not containing the origin form a co-minimal pair in $\bbZ^d$. 

\item 
If
$$t_n \geq 6 t_{n-1}
\quad
\text{ for all } 
n \geq 1,$$
then the set 
$\calW$
and a subset $\calG$ of $\bbZ^d\setminus \bbZ^d_{\geq 0}$ form a co-minimal pair in $\bbZ^d$.

\item 
If
$$t_n > 3 t_{n-1}
\quad
\text{ for all } 
n \geq 1,$$
then the set 
$\calV$
and a symmetric subset $\calP$ of $\bbZ^d$ containing the origin form a co-minimal pair in $\bbZ^d$. 
\item 
If
$$t_n > 3 t_{n-1} 
\quad
\text{ for all } 
n \geq 1,$$
then the set 
$\calV$
and a symmetric subset $\calQ$ of $\bbZ^d$ not containing the origin form a co-minimal pair in $\bbZ^d$. 
\item 
If
$$t_n \geq 6 t_{n-1}
\quad
\text{ for all } 
n \geq 1,$$
then the set 
$\calV$
and a subset $\calR$ of $\bbZ^d\setminus \bbZ^d_{\geq 0}$ form a co-minimal pair in $\bbZ^d$. 
\item 
If
$$t_n \geq 2 t_{n-1}
\quad
\text{ for all } 
n \geq 1,$$
then the set 
$\calT$
is a minimal complement to some subset of $\bbZ^d\setminus \bbZ^d_{\geq 0}$. 
In addition, if at least one coordinate of the sequence of points $\{t_n - 2t_{n-1}\}_{n\geq 1}$ goes to $\infty$, 
then the set 
$\calT$
and a subset $\calS$ of $\bbZ^d\setminus \bbZ^d_{\geq 0}$ form a co-minimal pair in $\bbZ^d$.

\end{enumerate}
\end{theorem}

\begin{theorem}
\label{Thm:UncountableSubsets}
Any infinite subset of any finitely generated abelian group has uncountably many subsets which is a part of a co-minimal pair. In particular, any infinite subset of $\bbZ^d$ has uncountably many subsets which admit minimal complements.
\end{theorem}

It turns out that there are plenty of subsets of $\bbZ^d$ each of which is a part of a co-minimal pair. 
\begin{theorem}
\label{Thm:Uncountable}
The set of integers $\bbZ^d$ has uncountably many 
\begin{enumerate}
\item 
symmetric subsets containing the origin, each of which forms a co-minimal pair together with a symmetric subset of $\bbZ^d$ containing the origin,
\item 
symmetric subsets containing the origin, each of which forms a co-minimal pair together with a symmetric subset of $\bbZ^d$ not containing the origin,
\item 
symmetric subsets containing the origin, each of which forms a co-minimal pair together with a subset of $\bbZ^d$ which avoids $\bbZ^d_{\geq 0}$,

\item 
symmetric subsets not containing the origin, each of which forms a co-minimal pair together with a symmetric subset of $\bbZ^d$ containing the origin,
\item 
symmetric subsets not containing the origin, each of which forms a co-minimal pair together with a symmetric subset of $\bbZ^d$ not containing the origin,
\item 
symmetric subsets not containing the origin, each of which forms a co-minimal pair together with a subset of $\bbZ^d$ which avoids $\bbZ^d_{\geq 0}$,

\item 
subsets contained in $\bbZ^d_{\geq 0}$, each of which forms a co-minimal pair together with a subset of $\bbZ^d$ which avoids $\bbZ^d_{\geq 0}$.
\end{enumerate}
\end{theorem}

\section{Proofs}
For two points $P, Q\in \bbZ^d$ satisfying $P\leq Q$, let $\calX_{P, Q}$ denote the subset of $\bbZ^d$ defined by 
$$
\calX_{P, Q}
=
\bbZ^d_{\geq P}
\setminus 
\bbZ^d_{\geq Q}
.
$$
\begin{lemma}
\label{Lemma:SumBound}
Let $P, Q$ be points in $\bbZ^d$ satisfying $P \leq Q$, and $A$ be a nonempty subset of $\bbZ^d$. For any $v\in \bbZ^d$, the inclusion 
$$\calX_{P, Q} + A_{\leq v} \subseteq 
\bbZ^d
\setminus 
\bbZ^d_{\geq Q+v} $$
holds. 
\end{lemma}

\begin{proof}
Note that 
\begin{align*}
\calX_{P, Q} + A_{\leq v} 
& \subseteq 
\cup _{a\in A , a \leq v} (\calX_{P, Q} +a)\\
& \subseteq 
\cup _{a\in A , a \leq v} \calX_{P+a, Q+a}\\
& \subseteq 
(\cup _{a\in A , a \leq v} \bbZ^d_{\geq P+a} )
\setminus 
(\cup _{a\in A , a \leq v} \bbZ^d_{\geq Q+a} )\\
& \subseteq 
(\cup _{a\in A , a \leq v} \bbZ^d_{\geq P+a} )
\setminus 
\bbZ^d_{\geq Q+v} \\
& \subseteq 
\bbZ^d
\setminus 
\bbZ^d_{\geq Q+v} .
\end{align*}
\end{proof}

Set 
$$t_{-2} = t_{-1} = 0.$$
Define the subsets $\{\calI_n\}_{n\geq 0}$ of $\bbZ^d$ as follows. 
$$\calI_n 
= 
\calX_{-t_n, -t_{n-1} }.
$$
Consider the subsets $\{\calJ_n\}_{n\geq 0}$ of $\bbZ^d$ defined by 
$$\calJ_n 
= 
\calX_{-t_n, -t_{n-1} -t_{n-2}}. 
$$
These subsets make sense since $t_n \geq 2 t_{n-1}$ holds for $n\geq 0$. 
Note that for $n\geq 0$,
$$
-t_{n-1} - t_{n-2}
\leq 
-t_{n-1}
$$
holds, this implies 
$$
\bbZ^d_{\geq -t_{n-1} - t_{n-2} }
\supseteq 
\bbZ^d_{\geq -t_{n-1}} ,
$$
which yields 
$$\calJ_n \subseteq \calI_n.$$
Since all the coordinates of $t_n$ are positive for some $n\geq 0$, we obtain $-t_n \to (-\infty, \cdots, -\infty)$. Thus it follows that 
$$\cup _{n\geq 0} \calI_n = \bbZ^d \setminus \bbZ^d_{\geq 0}.$$

\begin{proposition}
\label{Prop:M}
\quad
\begin{enumerate}
\item 
Each of the sets
$$
\cup _{m\geq n+2} \calJ_m + \calW
,
\calJ_{n+1} + (\calW\setminus \{t_n\})$$
contains no point of $\calI_n$ for all $n\geq 0$. 
\item
For any $n\geq 1$, the inclusion
$$
(\calJ_0 \cup \cdots \cup (\calJ_n\cap \bbZ^d_{\geq -2t_{n-1}})) + \calW 
\subseteq 
(\bbZ^d\setminus \bbZ^d_{\geq -t_n})
\cup 
\bbZ^d_{\geq -3t_{n-1}}
$$
holds. 
\item
If 
$$t_n > 3 t_{n-1}
\quad
\text{ for all } 
n \geq 1,$$
then the set 
$$
\cup _{m\geq n+1}(- (\calJ_m\cap \bbZ^d_{\geq -2t_{m-1} })) + \calW
$$
contains no point of $\calI_n$ for all $n\geq 0$.
\item
For any $n\geq 1$, the inclusion
$$
(-(\calJ_1 \cup \cdots \cup \calJ_n)) + (\calW \setminus \{-t_n\})
\subseteq 
\bbZ^d_{\leq -2t_n}
\cup 
(\bbZ^d \setminus \bbZ^d_{\leq -t_{n-1}})
$$
holds if 
$$t_n \geq 3 t_{n-1}
\quad
\text{ for all } 
n \geq 1.$$
\end{enumerate}
\end{proposition}

\begin{proof}
For $n\geq 0$, the inclusions 
\begin{align*}
\calJ_{m} + \calW_{\leq t_{m-2}}
& \subseteq 
\calX_{-t_m, - t_{m-1} - t_{m-2} } + \calW_{\leq t_{m-2}}\\
& \subseteq 
\bbZ^d \setminus \bbZ^d_{\geq - t_{m-1} - t_{m-2} + t_{m-2} }\\
& \subseteq 
\bbZ^d \setminus \bbZ^d_{\geq - t_{m-1}}\\
& \subseteq 
\bbZ^d \setminus \bbZ^d_{\geq - t_n}\\
\end{align*}
hold for $m\geq n+1$ (the second inclusion follows from Lemma \ref{Lemma:SumBound}), 
the inclusions 
\begin{align*}
\calJ_{m} + t_{m-1}
& \subseteq 
\calX_{-t_m, - t_{m-1} - t_{m-2} } + t_{m-1}\\
& = 
\calX_{-t_m+t_{m-1}, - t_{m-1} - t_{m-2} + t_{m-1}}\\
& = 
\calX_{-t_m+t_{m-1}, - t_{m-2} }\\
& \subseteq 
\calX_{-t_m+t_{m-1}, - t_n }
\end{align*}
hold for $m\geq n+2$, 
the inclusions
\begin{align*}
\calJ_{m} + \calW_{\geq t_m}
& \subseteq
\bbZ^d_{\geq - t_{m}} + \calW_{\geq t_m}\\
& =
\cup _{r\geq m} \bbZ^d_{\geq - t_{m} + t_r}\\
& \subseteq
\bbZ^d_{\geq - t_{m} + t_{m}}\\
& \subseteq 
\bbZ^d_{\geq 0}
\end{align*}
hold for $m\geq n+1$. This proves part (1). 

For $n\geq 1$, the inclusions
\begin{align*}
(\calJ_0 \cup \cdots \cup \calJ_n) + \calW_{\leq -t_n}
& \subseteq 
\calX_{-t_n, 0} +\calW_{\leq -t_n}\\
& \subseteq 
\bbZ^d\setminus \bbZ^d_{\geq - t_n}
\end{align*}
hold (the second inclusion follows from Lemma \ref{Lemma:SumBound}), 
the inclusions
\begin{align*}
(\calJ_0 \cup \cdots \cup (\calJ_n\cap \bbZ^d_{\geq -2t_{n-1}})) + \calW_{\geq -t_{n-1}}
& \subseteq 
\bbZ^d_{\geq -2t_{n-1}  } + \calW_{\geq -t_{n-1}}\\
& \subseteq 
\bbZ^d_{\geq -3t_{n-1}}
\end{align*}
hold for $r\leq n-1$. This proves part (2). 

For $n\geq 0$, $m\geq n+1$, the inclusions 
\begin{align*}
(- (\calJ_m\cap \bbZ^d_{\geq -2t_{m-1} })) - t_r
& \subseteq 
\bbZ^d_{\leq 2t_{m-1}  } - t_r \\
& \subseteq 
\bbZ^d_{\leq 2t_{m-1}  -t_m }\\
& \subseteq 
\bbZ^d_{< - t_{m-1}} \\
& \subseteq 
\bbZ^d_{< - t_n } 
\end{align*}
hold for $r\geq m$, 
the inclusions
\begin{align*}
(-\calJ_m) - t_r
& \subseteq 
(
\bbZ^d_{\leq t_m} 
\setminus 
\bbZ^d_{\leq t_{m-1}}
)- t_r \\
& \subseteq 
(
\bbZ^d
\setminus 
\bbZ^d_{\leq t_{m-1}}
)- t_r \\
& = 
\bbZ^d
\setminus 
\bbZ^d_{\leq t_{m-1}-t_r}\\
& \subseteq 
\bbZ^d
\setminus 
\bbZ^d_{\leq t_{m-1}-t_{m-1}}\\
& \subseteq 
\bbZ^d
\setminus 
\bbZ^d_{\leq 0}
\end{align*}
hold for $r\leq m-1$. This proves part (3). 

For $n\geq 1$, the inclusions 
\begin{align*}
 -(\calJ_1 \cup \cdots \cup \calJ_n) - t_r
& \subseteq 
\bbZ^d \setminus \bbZ^d_{\leq t_0} -t_r\\
& = 
\bbZ^d \setminus \bbZ^d_{\leq t_0 -t_r}\\
& \subseteq 
\bbZ^d \setminus \bbZ^d_{\leq t_0 -t_{n-1}}\\
& \subseteq 
\bbZ^d \setminus \bbZ^d_{\leq -t_{n-1}}
\end{align*}
hold for $r\leq n-1$, 
the inclusions 
\begin{align*}
 -(\calJ_1 \cup \cdots \cup \calJ_n) - t_r
& \subseteq 
\bbZ^d_{\leq t_n } - t_r\\
& \subseteq 
\bbZ^d_{\leq t_n -  t_{n+1}}\\
& \subseteq 
\bbZ^d_{\leq -2t_n }\\
\end{align*}
hold for $r\geq n+1$. This proves part (4). 
\end{proof}

\begin{lemma}
\label{Lemma:Finiteness}
Let $S$ and $T$ be nonempty subsets of an abelian group $G$ such that $S + T = G$. If the set $S$ is countable, and each element of $G$ can be expressed as a sum of an element of $S$ and an element of $T$ only in finitely many ways, then some nonempty subset of $S$ is a minimal complement to $T$. 
\end{lemma}

\begin{proof}
The lemma follows when $S$ is finite. 

Let us consider the case when $S$ is infinite. Let $s_1, s_2, \cdots$ be elements of $G$ such that $S = \{s_1, s_2, \cdots\}$. Define $S_1 = S$ and for each positive integer $i\geq 1$, define 
$$S_{i+1} 
: = 
\begin{cases}
S_i \setminus \{-s_i\} & \text{ if $S_i \setminus \{-s_i\}$ is a complement to $T$,}\\
S_i & \text{ otherwise.}
\end{cases}
$$
Let $\calS$ denote the subset $\cap _{i \geq 1} S_i$ of $G$. We claim that the set $\calS$ is a minimal complement to $T$.

For each $y\in G$ and for each $i\geq 1$, there exist elements $s_{y, i}\in S_i , t_{y, i}\in T$ such that $$y = s_{y, i} + t_{y, i}.$$ Consequently, for some element $s_y \in S$, we have the equality $s_y = s_{y, i}$  for infinitely many $i$. Further, for such integers $i$, we have $t_y = t_{y, i}$ where $t_y: = y - s_y\in T$. This implies that $y - t_y = s_{y, i}$ holds for infinitely many $i$. Hence, for each integer $i\geq 1$, there exists an integer $\ell_i \geq i$ such that $y - t_y = s_{y, \ell_i}$, which yields $$y\in t_y  + S_{\ell_i} \subseteq t_y + S_i.$$ As a consequence, $y$ lies in $t_y +  S$, i.e., $y\in S + T$. Hence $\calS$ is an additive complement of $T$. It follows that $\calS$ is a minimal complement to $T$. 

\end{proof}

\begin{proof}[Proof of Theorem \ref{Thm}(1)]
Define the subsets $\calE_n$ of $\bbZ^d$ for $n\geq 0$ as follows. 
$$\calE_n 
= 
\begin{cases}
\{0\} & \text{ if } n = -1,\\
\emptyset & \text{ if } n = 0,\\
\{
- t_{n-1}
\}
+ 
(\calI_{n-1} \setminus ( (\cup_{-1\leq m \leq n-1} (\calE_m \cup (-\calE_m))) + \calW))
& \text{ if } n \geq 1. 
\end{cases}
$$
Define the subset $\calE$ of $\bbZ^d$ by 
$$\calE : = \cup _{n\geq -1} (\calE_n \cup (-\calE_n)).$$
For $n\geq 0$, 
the inclusions
\begin{align*}
\calE + \calW 
& \supseteq 
(\calE_{n+1} + t_n) \cup 
( (\cup_{-1\leq m \leq n} (\calE_m \cup (-\calE_m))) + \calW)\\
& \supseteq 
(\calI_{n} \setminus ( (\cup_{-1\leq m \leq n} (\calE_m \cup (-\calE_m))) + \calW))\cup 
( (\cup_{-1\leq m \leq n} (\calE_m \cup (-\calE_m))) + \calW)\\
& \supseteq 
\calI_n
\end{align*}
hold and hence $-\calI_n \subseteq (-\calE) + (-\calW) = \calE + \calW$. 
Moreover, the inclusions
$$
\calE + \calW 
\supseteq 
\calE_{-1} + \calW 
\supseteq \{0\}$$
hold. It follows that $\calW$ is an additive complement to $\calE$. 

We claim that $\calW$ is a minimal complement of $\calE$.
Since $t_n > 3t_{n-1}$ for $n\geq 1$, from Proposition \ref{Prop:M}, it follows that for $n\geq 0$, no point of $\calE \times \calW$ other than $(0, -t_n)$ goes to $-t_n$ under the addition map $\calE \times \calW \to \bbZ^d$, and hence no point of $\calE \times \calW$ other than $(0, t_n)$ goes to $t_n$ under the addition map $\calE \times \calW \to \bbZ^d$. Thus $\calW$ is a minimal complement of $\calE$. 

We claim that $\calE$ is a minimal complement to $\calW$. 
On the contrary, let us assume that $\calE$ is not a minimal complement to $\calW$. Hence $\calE\setminus \{e\}$ is an additive complement to $\calW$ for some $e\in \calE$. 
Note that $e \neq 0$. 
Since $\calE$ is symmetric, we may assume that $e$ lies in $\calE_{n+1}$ for some $n\geq 0$. Thus $t_n +e$ lies in $\calI_n$. 
It follows from Proposition \ref{Prop:M} that no element of $\calI_n$ lies in 
$$
((\cup _{m\geq n+2} \calE_m) + \calW)
\cup 
((\cup _{m\geq n+1} (-\calE_m)) + \calW)
\cup 
(\calE_{n+1} + (\calW \setminus \{t_n\})).$$
So
$t_n+e$ belongs to 
$(
(\cup_{0 \leq m \leq n} (\calE_m \cup (-\calE_m)))
+ \calW
)
\cup 
((\calE_{n+1} \setminus\{e\}) + \{t_n\})$. 
Since $e\in \calI_{n+1}$, it follows that 
$t_n + e$ lies in $(\calE_{n+1} \setminus\{e\}) + \{t_n\}$, which yields $e\in \calE_{n+1} \setminus\{e\}$. This contradicts the hypothesis that $\calE$ is not a minimal complement of $\calW$. 
\end{proof}

\begin{proof}[Proof of Theorem \ref{Thm}(2)]
Define the subsets $\calF_n$ of $\bbZ^d$ for $n\geq 0$ as follows. 
$$\calF_n 
= 
\begin{cases}
\{-t_0\} & \text{ if } n = 0,\\
\{
- t_{n-1}
\}
+ 
(\calI_{n-1} \setminus ( (\cup_{0\leq m \leq n-1} (\calF_m \cup (-\calF_m))) + \calW))
& \text{ if } n \geq 1. 
\end{cases}
$$
Define the subset $\calF$ of $\bbZ^d$ by 
$$\calF : = \cup _{n\geq 0} (\calF_n \cup (-\calF_n)).$$
For $n\geq 0$, 
the inclusions
\begin{align*}
\calF + \calW 
& \supseteq 
(\calF_{n+1} + t_n) \cup 
( (\cup_{0\leq m \leq n} (\calF_m \cup (-\calF_m))) + \calW)\\
& \supseteq 
(\calI_{n} \setminus ( (\cup_{0\leq m \leq n} (\calF_m \cup (-\calF_m))) + \calW))\cup 
( (\cup_{0\leq m \leq n} (\calF_m \cup (-\calF_m))) + \calW)\\
& \supseteq 
\calI_n
\end{align*}
hold and hence $-\calI_n \subseteq (-\calF) + (-\calW) = \calF + \calW$. 
Moreover, the inclusions
$$
\calF + \calW 
\supseteq 
\calF_{0} + \calW 
\supseteq \{0\}$$
hold. It follows that $\calW$ is an additive complement to $\calF$. 

We claim that $\calW$ is a minimal complement of $\calF$.
Since $t_n > 3t_{n-1}$ for $n\geq 1$, from Proposition \ref{Prop:M}, it follows that for $n\geq 0$, no point of $\calF \times \calW$ other than $(-2t_n, t_n)$ goes to $-t_n$ under the addition map $\calF \times \calW \to \bbZ^d$, and hence no point of $\calF \times \calW$ other than $(2t_n, -t_n)$ goes to $t_n$ under the addition map $\calF \times \calW \to \bbZ^d$. Thus $\calW$ is a minimal complement of $\calF$. 

We claim that $\calF$ is a minimal complement to $\calW$. 
On the contrary, let us assume that $\calF$ is not a minimal complement to $\calW$. Hence $\calF\setminus \{f\}$ is an additive complement to $\calW$ for some $f\in \calF$. 
Since $\calF$ is symmetric, we may assume that $f$ lies in $\calF_{n+1}$ for some $n\geq -1$. 
Since no point of $\calF \times \calW$ other than $(-t_0, -t_0)$ goes to $-2t_0$ under the addition map $\calF \times \calW \to \bbZ^d$, it follows that $f \neq -t_0$, i.e., $f\notin \calF_0$. 
So $f$ lies in $\calF_{n+1}$ for some $n\geq 0$. Thus $t_n +f$ lies in $\calI_n$. 
It follows from Proposition \ref{Prop:M} that no element of $\calI_n$ lies in 
$$
((\cup _{m\geq n+2} \calF_m) + \calW)
\cup 
((\cup _{m\geq n+1} (-\calF_m)) + \calW)
\cup 
(\calF_{n+1} + (\calW \setminus \{t_n\})).$$
So
$t_n+f$ belongs to 
$(
(\cup_{0 \leq m \leq n} (\calF_m \cup (-\calF_m)))
+ \calW
)
\cup 
((\calF_{n+1} \setminus\{f\}) + \{t_n\})$. 
Since $f\in \calI_{n+1}$, it follows that 
$t_n + f$ lies in $(\calF_{n+1} \setminus\{f\}) + \{t_n\}$, which yields $f\in \calF_{n+1} \setminus\{f\}$. This contradicts the hypothesis that $\calF$ is not a minimal complement of $\calW$. 
\end{proof}

\begin{proof}[Proof of Theorem \ref{Thm}(3)]
Let $\{m_k\}_{k\geq 0}$ be an increasing sequence such that $m_0 \geq 2$ and $t_{m_k -2}\geq (k+1)\bfone$ for all $k\geq 0$. We define the subsets $G_n$ of $\bbZ^d$ for $n\geq 0$ as follows. 
$$G_n 
= 
\begin{cases}
\calI_0 & \text{ if } n = 0, \\
\calX_{-2t_{n-1}, -t_{n-1} - t_{n-2}}
& \text{ if } n \geq 1 \text{ and } n \neq m_k \text{ for all } k \geq 0,\\
\calX_{- t_n + k\bfone, - t_n + (k+1)\bfone}
\cup 
\calX_{-2t_{n-1}, -t_{n-1} - t_{n-2}}
& \text{ if } n \geq 1 \text{ and } n = m_k \text{ for some } k \geq 0.\\
\end{cases}
$$
Let $G$ denote the union $\cup_{n\geq 0 } G_n$. Note that $\calT$ is an additive complement of $G$. Indeed, the inclusions
\begin{align*}
G + \calT 
& \supseteq 
\left(\cup_{n\geq 1}  (\calX_{-2t_{n-1}, -t_{n-1} - t_{n-2}}+ \calT) \right)
\bigcup 
\left(\cup_{k\geq 0} (\calX_{- t_{m_k} + k\bfone, - t_{m_k} + (k+1)\bfone}+ \calT)\right) \\
& \supseteq 
(\cup_{n\geq 1} \calI_{n-1} )
\cup 
\bbZ^d_{\geq 0}  \\
& = \bbZ^d
\end{align*}
hold, which shows that $G + \calT = \bbZ^d$. So $\calW$ is an additive complement to $G$.

We claim that $\calW$ is a minimal complement of $G \setminus (\{-2t_0\}\cup \{- t_n - 3t_{n-1} \,|\,n \geq 1\})$. For $k\geq 0$, the inequalities 
\begin{align*}
t_{m_k} - t_{m_k-1} - t_{m_k-2} 
& = 
(t_{m_k} - 2t_{m_k-1}) + (t_{m_k-1} - 2t_{m_k-2})  + t_{m_k-2} \\
& \geq t_{m_k-2} \\
& \geq (k+1)\bfone
\end{align*}
hold, which implies 
$$-t_{m_k}  +  (k+1)\bfone
\leq 
 - t_{m_k-1} - t_{m_k-2} ,$$
this yields 
$$
\bbZ^d_{\geq -t_{m_k}  +  (k+1)\bfone}
\supseteq 
\bbZ^d_{\geq  - t_{m_k-1} - t_{m_k-2} },$$
and hence 
$$
\calX_{-t_{m_k}, -t_{m_k}  +  (k+1)\bfone}
\subseteq 
\calX_{-t_{m_k}, - t_{m_k-1} - t_{m_k-2} }.$$
Thus $G_n \subseteq \calJ_n$ for all $n\geq 0$. Let $n$ be a positive integer. Note that $-3t_{n-1}, -4t_{n-1}$ lie in $\calI_n$. The inclusions 
\begin{align*}
\cup_{0 \leq m < n} G_m + \calW_{\leq -t_n}
& \subseteq 
(\bbZ^d\setminus \bbZ^d_{\geq 0}) + \bbZ^d_{\leq - t_n} \\
& \subseteq 
\bbZ^d\setminus \bbZ^d_{\geq -t_n}\\
& \subseteq 
\bbZ^d\setminus \bbZ^d_{\geq - 4t_{n-1}  }
\end{align*}
hold, and the inclusions 
\begin{align*}
\cup_{0 \leq m < n} G_m + \calW_{\geq  -t_{n-1}}
& \subseteq 
\bbZ^d_{\geq - t_{n-1} } + \bbZ^d_{\geq - t_{n-1}}\\
& \subseteq 
\bbZ^d_{\geq - 2t_{n-1} } \\
\end{align*}
hold. Using Proposition \ref{Prop:M}, it follows that the set 
$$
\cup _{m \neq n, n+1} G_m + \calW$$
contains none of $-3t_{n-1}, -4t_{n-1}$. 
Since $-3t_{n-1} , -4 t_{n-1}\in \calI_n$, and $G_{n+1} + t_n$ contains $\calI_n$, it follows that $- t_n - 3t_{n-1}, -t_n -4t_{n-1} \in G_{n+1}$. By Proposition \ref{Prop:M}, $G_{n+1} + (\calW \setminus \{t_n\})$ contains no element of $\calI_n$. So no point of $\cup _{m \neq n} G_m + \calW$ other than $(- t_n - 3t_{n-1}, t_n)$ goes to $-3t_{n-1}$, and no point of $\cup _{m \neq n} G_m + \calW$ other than $(- t_n - 4t_{n-1}, t_n)$ goes to $-4t_{n-1}$. Note that the inclusions
\begin{align*}
(\{-t_{n-1}\} + \calI_{n-1} )+ \calW_{> -t_{n-1}}
& \subseteq 
\bbZ^d_{\geq -2t_{n-1} } + \bbZ^d_{> -t_{n-1}}\\
& \subseteq 
\bbZ^d_{> -2t_{n-1} -t_{n-1}}\\
& = 
\bbZ^d_{> -3t_{n-1} }
\end{align*}
hold, the inclusions
\begin{align*}
(\{-t_{n-1}\} + \calI_{n-1} )+ \calW_{\leq -t_{n}}
& \subseteq 
\calX_{-t_{n-1}, -t_{n-2}} +  \calW_{\leq -t_{n}}\\
& \subseteq 
\calX_{-t_{n-1}, 0} +  \calW_{\leq -t_{n}}\\
& \subseteq 
\bbZ^d \setminus \bbZ^d_{\geq -t_n}
\end{align*}
hold. Moreover, the inclusion 
$$
\calX_{- t_n + k\bfone, - t_n + (k+1)\bfone} + \calW_{\geq t_n} \subseteq \bbZ^d_{\geq 0} 
$$
holds when $n = m_k$ for some $k\geq 0$, and the inclusions 
\begin{align*}
\calX_{- t_n + k\bfone, - t_n + (k+1)\bfone}  + \calW_{\leq t_{n-1}}
& \subseteq 
\bbZ^d\setminus \bbZ^d_{\geq - t_n +  (k+1)\bfone + t_{n-1}} \\
& \subseteq 
\bbZ^d\setminus \bbZ^d_{\geq - 6t_{n-1} +  (k+1)\bfone + t_{n-1}} \\
& = 
\bbZ^d\setminus \bbZ^d_{\geq - 5t_{n-1} +  (k+1)\bfone} \\
& \subseteq 
\bbZ^d\setminus \bbZ^d_{\geq - 4t_{n-1}  - t_{n-2} +  (k+1)\bfone} \\
& \subseteq 
\bbZ^d\setminus \bbZ^d_{\geq - 4t_{n-1}  } 
\end{align*}
hold when $n = m_k$ for some $k\geq 0$. Also note that $G_n$ contains $- 2t_{n-1}$. It follows that no element of $G_n \times \calW$ other than $(-2t_{n-1}, -t_{n-1})$ goes to $-3t_{n-1}$ under the addition map $G_n\times \calW \to \bbZ^d$. Hence no element of $G \times \calW$ other than $(-2t_{n-1}, -t_{n-1}), (- t_n - 3t_{n-1}, t_n)$ goes to $-3t_{n-1}$ under the addition map $G\times \calW \to \bbZ^d$. Moreover, $G_n + \calW$ does not contain $-4t_{n-1}$. Hence no element of $G \times \calW$ other than $(- t_n - 4t_{n-1}, t_n)$ goes to $-4t_{n-1}$ under the addition map $G\times \calW \to \bbZ^d$. Proposition \ref{Prop:M} implies that 
$$ - t_0
\notin
(\cup_{m \geq 2}  G_m + \calW)
\cup 
(G_{1} + (\calW \setminus \{t_{0}\})).
$$
Also note that the inclusions
\begin{align*}
G_0 + \calW_{\geq t_0}
& \subseteq 
\calX_{-t_0, 0} + \calW_{\geq t_0}\\
& \subseteq 
\bbZ^d_{\geq -t_0} + \calW_{\geq t_0}\\
& \subseteq 
\bbZ^d_{\geq 0}
\end{align*}
hold and the inclusions
\begin{align*}
G_0 + \calW_{\leq -t_0}
& \subseteq 
\calX_{-t_0, 0} + \calW_{\leq -t_0}\\
& \subseteq 
\bbZ^d \setminus \bbZ^d_{\geq -t_0}
\end{align*}
hold, and hence no element of $G\times \calW$ other than $(-2t_0, t_0), (-t_0, 0)$ goes to $-t_0$ under the addition map $G\times \calW \to \bbZ^d$. It follows that $\calW$ is a minimal complement to $G \setminus (\{-2t_0\}\cup \{- t_n - 3t_{n-1} \,|\,n \geq 1\})$. 

We claim that $\calW$ and some subset of $G \setminus (\{-2t_0\}\cup \{- t_n - 3t_{n-1} \,|\,n \geq 1\})$ form a co-minimal pair. By Proposition \ref{Prop:M}, each element of $\bbZ^d\setminus \bbZ^d_{\geq 0} = \cup_{n\geq 0} \calI_n$ can be expressed as a sum of an element of $G$ and an element of $\calW$ only in finitely many ways. Note that the inclusions 
\begin{align*}
G_m + \calW_{\leq t_{m-1}}
& \subseteq 
\calX_{-t_m, - t_{m-1} } + \calW_{\leq t_{m-1}}\\
& \subseteq  
\bbZ^d\setminus \bbZ^d_{\geq 0}
\end{align*}
hold for $m\geq 1$ and $r\leq m-1$, the inclusions 
\begin{align*}
(\{-t_{m-1}\} + \calI_{m-1} ) + \calW_{\geq  t_m} 
& \subseteq
\bbZ^d_{\geq - 2t_{m-1} } + \bbZ^d_{\geq  t_m}\\
& \subseteq
\bbZ^d_{\geq t_m- 2t_{m-1} }
\end{align*}
hold for $m\geq 1$,
the inequality
$$
\calX_{-t_{m_k} + k\bfone, -t_{m_k} + (k+1)\bfone } + t_r
\geq 
\bbZ^d_{\geq - t_{m_k} + k\bfone + t_r} 
= 
\bbZ^d_{\geq k\bfone}
$$
holds for $k\geq 0$ and for $r\geq m_k$, and hence 
$$\cup_{m\geq M} G_m + \calW$$
does not contain a given point of $\bbZ^d_{\geq 0}$ for some large enough $M$. 
Hence any given element element of $\bbZ^d_{\geq 0}$ can be expressed as a sum of an element of $G$ and an element of $\calW$ only in finitely many ways. So each element of $\bbZ^d$ can be expressed as a sum of an element of $G$ and an element of $\calW$ only in finitely many ways. In particular, each element of $\bbZ^d$ can be expressed as a sum of an element of $G \setminus (\{-2t_0\}\cup \{- t_n - 3t_{n-1} \,|\,n \geq 1\})$ and an element of $\calW$ only in finitely many ways. By Lemma \ref{Lemma:Finiteness}, it follows that some nonempty subset $\calG$ of $G \setminus (\{-2t_0\}\cup \{- t_n - 3t_{n-1} \,|\,n \geq 1\})$ is a minimal complement to $\calW$. Since $\calW$ is a minimal complement to $G \setminus (\{-2t_0\}\cup \{- t_n - 3t_{n-1} \,|\,n \geq 1\})$, it follows that $\calW$ is a minimal complement to $\calG$. Hence $(\calG, \calW)$ is a co-minimal pair. 
\end{proof}

\begin{proof}[Proof of Theorem \ref{Thm}(4)]
Define the subsets $\calP_n$ of $\bbZ^d$ for $n\geq 0$ as follows. 
$$\calP_n 
= 
\begin{cases}
\{0\} & \text{ if } n = -1,\\
\{-t_0\} & \text{ if } n = 0,\\
\{
- t_{n-1}
\}
+ 
(\calI_{n-1} \setminus ( (\cup_{-1\leq m \leq n-1} (\calP_m \cup (-\calP_m))) + \calV))
& \text{ if } n \geq 1. 
\end{cases}
$$
Define the subset $\calP$ of $\bbZ^d$ by 
$$\calP : = \cup _{n\geq -1} (\calP_n \cup (-\calP_n)).$$
For $n\geq 1$, 
the inclusions
\begin{align*}
\calP + \calV 
& \supseteq 
(\calP_{n+1} + t_n) \cup 
( (\cup_{-1\leq m \leq n} (\calP_m \cup (-\calP_m))) + \calV)\\
& \supseteq 
(\calI_{n} \setminus ( (\cup_{-1\leq m \leq n} (\calP_m \cup (-\calP_m))) + \calV))\cup 
( (\cup_{-1\leq m \leq n} (\calP_m \cup (-\calP_m))) + \calV)\\
& \supseteq 
\calI_n
\end{align*}
hold and hence $-\calI_n \subseteq (-\calP) + (-\calV) = \calP + \calV$. 
Moreover, the inclusions
$$\calI_0\subseteq \calP_1 + t_0, $$
$$-\calI_0\subseteq (-\calP_1) + (- t_0), $$
$$
\calP + \calV 
\supseteq 
\calP_{0} + \calV 
\supseteq \{0\}$$
hold. It follows that $\calV$ is an additive complement to $\calP$. 

We claim that $\calV$ is a minimal complement of $\calP$.
Since $t_n > 3t_{n-1}$ for $n\geq 1$, from Proposition \ref{Prop:M}, it follows that for $n\geq 0$, no point of $\calP \times \calV$ other than $(0, -t_n)$ goes to $-t_n$ under the addition map $\calP \times \calV \to \bbZ^d$, and hence no point of $\calP \times \calV$ other than $(0, t_n)$ goes to $t_n$ under the addition map $\calP \times \calV \to \bbZ^d$. Thus $\calV$ is a minimal complement of $\calP$. 

We claim that $\calP$ is a minimal complement to $\calV$. 
On the contrary, let us assume that $\calP$ is not a minimal complement to $\calV$. Hence $\calP\setminus \{p\}$ is an additive complement to $\calV$ for some $p\in \calP$. 
Note that $p \neq 0$. 
Since $\calP$ is symmetric, we may assume that $p\in \calP_{n+1}$ for some $n\geq -1$.
Since no point of $\calP \times \calV$ other than $(-t_0, -t_0)$ goes to $-2t_0$ under the addition map $\calP \times \calV \to \bbZ^d$, it follows that $p \neq -t_0$, i.e., $p\notin \calP_0$. 
So $p$ lies in $\calP_{n+1}$ for some $n\geq 0$. Thus $t_n +p$ lies in $\calI_n$. 
It follows from Proposition \ref{Prop:M} that no element of $\calI_n$ lies in 
$$
((\cup _{m\geq n+2} \calP_m) + \calV)
\cup 
((\cup _{m\geq n+1} (-\calP_m)) + \calV)
\cup 
(\calP_{n+1} + (\calV \setminus \{t_n\})).$$
So
$t_n+p$ belongs to 
$(
(\cup_{0 \leq m \leq n} (\calP_m \cup (-\calP_m)))
+ \calV
)
\cup 
((\calP_{n+1} \setminus\{p\}) + \{t_n\})$. 
Since $p\in \calI_{n+1}$, it follows that 
$t_n + p$ lies in $(\calP_{n+1} \setminus\{p\}) + \{t_n\}$, which yields $p\in \calP_{n+1} \setminus\{p\}$. This contradicts the hypothesis that $\calP$ is not a minimal complement of $\calV$. 
\end{proof}

\begin{proof}[Proof of Theorem \ref{Thm}(5)]
Define the subsets $\calQ_n$ of $\bbZ^d$ for $n\geq 0$ as follows. 
$$\calQ_n 
= 
\begin{cases}
\{-t_0\} & \text{ if } n = 0,\\
\{
- t_{n-1}
\}
+ 
(\calI_{n-1} \setminus ( (\cup_{0\leq m \leq n-1} (\calQ_m \cup (-\calQ_m))) + \calV))
& \text{ if } n \geq 1. 
\end{cases}
$$
Define the subset $\calQ$ of $\bbZ^d$ by 
$$\calQ : = \cup _{n\geq 0} (\calQ_n \cup (-\calQ_n)).$$
For $n\geq 0$, 
the inclusions
\begin{align*}
\calQ + \calV 
& \supseteq 
(\calQ_{n+1} + t_n) \cup 
( (\cup_{0\leq m \leq n} (\calQ_m \cup (-\calQ_m))) + \calV)\\
& \supseteq 
(\calI_{n} \setminus ( (\cup_{0\leq m \leq n} (\calQ_m \cup (-\calQ_m))) + \calV))\cup 
( (\cup_{0\leq m \leq n} (\calQ_m \cup (-\calQ_m))) + \calV)\\
& \supseteq 
\calI_n
\end{align*}
hold and hence $-\calI_n \subseteq (-\calQ) + (-\calV) = \calQ + \calV$. 
Moreover, the inclusions
$$
\calQ + \calV 
\supseteq 
\calQ_{0} + \calV 
\supseteq \{0\}$$
hold. It follows that $\calV$ is an additive complement to $\calQ$. 

We claim that $\calV$ is a minimal complement of $\calQ$.
Since $t_n > 3t_{n-1}$ for $n\geq 1$, from Proposition \ref{Prop:M}, it follows that for $n\geq 0$, no point of $\calQ \times \calV$ other than $(-2t_n, t_n)$ goes to $-t_n$ under the addition map $\calQ \times \calV \to \bbZ^d$, and hence no point of $\calQ \times \calV$ other than $(2t_n, -t_n)$ goes to $t_n$ under the addition map $\calQ \times \calV \to \bbZ^d$. Thus $\calV$ is a minimal complement of $\calQ$. 

We claim that $\calQ$ is a minimal complement to $\calV$. 
On the contrary, let us assume that $\calQ$ is not a minimal complement to $\calV$. Hence $\calQ\setminus \{q\}$ is an additive complement to $\calV$ for some $q\in \calQ$. 
Since $\calQ$ is symmetric, we may assume that $q\in \calQ_{n+1}$ for some $n\geq -1$.
Since no point of $\calQ \times \calV$ other than $(-t_0, -t_0)$ goes to $-2t_0$ under the addition map $\calQ \times \calV \to \bbZ^d$, it follows that $q \neq -t_0$, i.e., $q\notin \calQ_0$. 
So $q$ lies in $\calQ_{n+1}$ for some $n\geq 0$. Thus $t_n +q$ lies in $\calI_n$. 
It follows from Proposition \ref{Prop:M} that no element of $\calI_n$ lies in 
$$
((\cup _{m\geq n+2} \calQ_m) + \calV)
\cup 
((\cup _{m\geq n+1} (-\calQ_m)) + \calV)
\cup 
(\calQ_{n+1} + (\calV \setminus \{t_n\})).$$
So
$t_n+q$ belongs to 
$(
(\cup_{0 \leq m \leq n} (\calQ_m \cup (-\calQ_m)))
+ \calV
)
\cup 
((\calQ_{n+1} \setminus\{q\}) + \{t_n\})$. 
Since $q\in \calI_{n+1}$, it follows that 
$t_n + q$ lies in $(\calQ_{n+1} \setminus\{q\}) + \{t_n\}$, which yields $q\in \calQ_{n+1} \setminus\{q\}$. This contradicts the hypothesis that $\calQ$ is not a minimal complement of $\calV$. 
\end{proof}

\begin{proof}[Proof of Theorem \ref{Thm}(6)]
Let $G$ be as in the proof of Theorem \ref{Thm}(3). 
Since $G + \calT = \bbZ^d$, it follows that $\calV$ is an additive complement to $G$.

We claim that $\calV$ is a minimal complement of $G \setminus \{- t_n - 3t_{n-1} \,|\,n \geq 1\}$. For $n\geq 1$, no element of $G \times \calW$ other than $(-2t_{n-1}, -t_{n-1}), (- t_n - 3t_{n-1}, t_n)$ goes to $-3t_{n-1}$ under the addition map $G\times \calW \to \bbZ^d$, and no element of $G \times \calW$ other than $(- t_n - 4t_{n-1}, t_n)$ goes to $-4t_{n-1}$ under the addition map $G\times \calW \to \bbZ^d$. Moreover, no element of $G\times \calW$ other than $(-2t_0, t_0), (-t_0, 0)$ goes to $-t_0$ under the addition map $G\times \calW \to \bbZ^d$. It follows that $\calV$ is a minimal complement to $G \setminus \{- t_n - 3t_{n-1} \,|\,n \geq 1\}$. 

Each element of $\bbZ^d$ can be expressed as a sum of an element of $G$ and an element of $\calW$ only in finitely many ways. In particular, each element of $\bbZ^d$ can be expressed as a sum of an element of $G \setminus \{- t_n - 3t_{n-1} \,|\,n \geq 1\}$ and an element of $\calV$ only in finitely many ways. By Lemma \ref{Lemma:Finiteness}, it follows that some nonempty subset $\calR$ of $G \setminus \{- t_n - 3t_{n-1} \,|\,n \geq 1\}$ is a minimal complement to $\calV$. Since $\calV$ is a minimal complement to $G \setminus \{- t_n - 3t_{n-1} \,|\,n \geq 1\}$, it follows that $\calV$ is a minimal complement to $\calR$. Hence $(\calR, \calV)$ is a co-minimal pair. 
\end{proof}

\begin{proof}[Proof of Theorem \ref{Thm}(7)]
Let $G$ be the set as in the proof of Theorem \ref{Thm}(3). 
Note that $\calT$ is an additive complement of $G$. 
Let $n\geq 0$ be an integer. 
Proposition \ref{Prop:M} implies that 
$$ - t_n
\notin
(\cup_{m \geq n+2}  G_m + \calT)
\cup 
(G_{n+1} + (\calT \setminus \{t_{n}\})).
$$
Note that the inclusions 
\begin{align*}
G_m + \calT
& \subseteq 
\bbZ^d_{\geq - t_m } + \bbZ^d_{\geq t_0} \\
& \subseteq 
\bbZ^d_{\geq - t_m + t_0}\\
& \subseteq 
\bbZ^d_{\geq -t_{n} + t_0} 
\end{align*}
hold for any $m\leq n$.
So 
$$ - t_n
\notin
(\cup_{m \neq n, n+1}  G_m + \calT)
\cup 
(G_{n+1} + (\calT \setminus \{t_{n}\})).
$$
Hence $\calT$ is a minimal complement of $G$. 

Since each element of $\bbZ^d$ can be expressed as a sum of an element of $G$ and an element of $\calV$ only in finitely many ways, it follows that each element of $\bbZ^d$ can be expressed as a sum of an element of $G$ and an element of $\calT$ only in finitely many ways. 
By Lemma \ref{Lemma:Finiteness}, some nonempty subset $\calS$ of $G$ is a minimal complement to $\calT$. Since $\calT$ is a minimal complement to $G$, it follows that $(\calS, \calT)$ is a co-minimal pair. 
\end{proof}

\begin{proof}
[Proof of Theorem \ref{Thm:UncountableSubsets}]
Let $G$ be a finitely generated abelian group. Then $G$ is isomorphic to the direct product of a finite group $G_\tors$ and a free abelian group $\bbZ^d$. Given any infinite subset $X$ of $G$, it contains an infinite subset $Y$ such that the projections of all the elements of $Y$ to $G_\tors$ are equal. 
It suffices to show that any infinite subset of $\bbZ^d$ has uncountably many subsets which admit minimal complements. Let $X$ be an infinite subset of $\bbZ^d$. Let $S$ be the subset consisting of the integers $1\leq i\leq d$ such that the absolute values of the $i$-th coordinate of the elements of $X$ form an unbounded set. Note that $S$ is nonempty. Thus $X$ has an infinite subset $Y$ such that the $i$-th coordinate of all the elements of $Y$ are equal for any $i\in \{1, 2, \cdots, d\}\setminus S$, and the absolute values of the $i$-th coordinate of the elements of $Y$ form an unbounded set for any $i\in S$. 
Replacing $X$ by one of its translate, we may assume that the $i$-th coordinates of the elements of $Y$ are equal to $0$ for any $i\in \{1, 2, \cdots, d\}\setminus S$.
Replacing $X$ by one of its image under an automorphism of $\bbZ^d$, we may assume that $Y$ has a subset $Z$ such that the $i$-th coordinates of the elements of $Z$ form an infinite subset of $\bbZ_{\geq 1}$ for any $i\in \{1, 2, \cdots, d\}\setminus S$.

Let $d'$ denote the cardinality of $S$. It suffices to show that if $A$ an infinite subset of $\bbZ^{d'}$ such that for any $1\leq i \leq d'$, the $i$-th coordinates of the points of $A$ form an infinite subset of $\bbZ_{\geq 1}$, then $A$ has uncountably many subsets which admit minimal complements.
Note that there is a sequence $\{x_n\}_{n\geq 0}$ contained in $A$ such that $x_n\geq 6x_{n-1}$ for all $n\geq 1$. For any subsequence $\{x_{n_k}\}$ of this sequence, the inequality $x_{n_k} \geq 6 x_{n_{k-1}}$ holds for any $k\geq 1$ and moreover, $x_{n_0} > 0$ holds and all the coordinates of any term of this subsequence are positive. By Theorem \ref{Thm}, the subset $\{x_{n_k} \,|\, k \geq 0\}$ of $A$ admits a minimal complement in $\bbZ^{d'}$. 
Since the finite subsets of $\bbN$ are precisely the bounded subsets of $\bbN$, it follows that the number of finite subsets of $\bbN$ is countable. Thus the number of infinite subsets of $\bbN$ is uncountable. So $\{x_n\}_{n\geq 0}$ has uncountably many subsequences. It follows that $A$ has uncountably many subsets which admit minimal complements.
\end{proof}

\begin{proof}
[Proof of Theorem \ref{Thm:Uncountable}]
Since the finite subsets of $\bbN$ are precisely the bounded subsets of $\bbN$, it follows that the number of finite subsets of $\bbN$ is countable. Thus the number of infinite subsets of $\bbN$ is uncountable. So any sequence has uncountably many subsequences. Note that if $\{x_n\}_{n\geq 0}$ is a sequence in $\bbZ^d$ such that $x_0 = (1, \cdots, 1)$ and 
$$x_n \geq 6 x_{n-1}
\quad
\text{ for all } 
n \geq 1,$$
then any subsequence of $\{x_n\}_{n\geq 0}$ satisfies the hypothesis of each of the seven parts of Theorem \ref{Thm}. Thus Theorem \ref{Thm:Uncountable} follows from the existence of a sequence $\{x_n\}_{n\geq 0}$ in $\bbZ^d$ such that $x_0 = (1, \cdots, 1)$ and 
$$x_n \geq 6 x_{n-1}
\quad
\text{ for all } 
n \geq 1,$$
which exists, for instance, consider the sequence $x_n = (6^n, \cdots, 6^n)$. 
\end{proof}

As an application of Theorem \ref{Thm}, we provide several examples of subsets of $\bbZ$ each of which is a part of a co-minimal pair. 

\begin{corollary}
For each of the following subsets $S_{1},S_{2}$, there exists subsets $S'_{1},S'_{2}$ such that $(S_{1},S'_{1})$ and $(S_{2},S'_{2})$ form co-minimal pairs.
\begin{enumerate}
\item $S_{1} := \{n^k \,|\, k \geq 0\}$ for any $n\geq 3$. 
\item $S_{2}:= \{2^k + k\,|\, k \geq 0\}$.
\end{enumerate}
\end{corollary}

\begin{proof}
	The growth condition of the elements of the subsets (considered as sequences) in this corollary clearly satisfies the condition of Theorem \ref{Thm}.
\end{proof}

\section{Acknowledgements}
The first author would like to thank the Department of Mathematics at the Technion where a part of the work was carried out. 
The second author would like to acknowledge the Initiation Grant from the Indian Institute of Science Education and Research Bhopal, and the INSPIRE Faculty Award from the Department of Science and Technology, Government of India.

% \bibliography{../../../2020BibShort/biblio}
% \bibliographystyle{amsalpha} 
\def\cprime{$'$} \def\Dbar{\leavevmode\lower.6ex\hbox to 0pt{\hskip-.23ex
  \accent"16\hss}D} \def\cfac#1{\ifmmode\setbox7\hbox{$\accent"5E#1$}\else
  \setbox7\hbox{\accent"5E#1}\penalty 10000\relax\fi\raise 1\ht7
  \hbox{\lower1.15ex\hbox to 1\wd7{\hss\accent"13\hss}}\penalty 10000
  \hskip-1\wd7\penalty 10000\box7}
  \def\cftil#1{\ifmmode\setbox7\hbox{$\accent"5E#1$}\else
  \setbox7\hbox{\accent"5E#1}\penalty 10000\relax\fi\raise 1\ht7
  \hbox{\lower1.15ex\hbox to 1\wd7{\hss\accent"7E\hss}}\penalty 10000
  \hskip-1\wd7\penalty 10000\box7}
  \def\polhk#1{\setbox0=\hbox{#1}{\ooalign{\hidewidth
  \lower1.5ex\hbox{`}\hidewidth\crcr\unhbox0}}}
\providecommand{\bysame}{\leavevmode\hbox to3em{\hrulefill}\thinspace}
\providecommand{\MR}{\relax\ifhmode\unskip\space\fi MR }
\providecommand{\MRhref}[2]{%
  \href{http://www.ams.org/mathscinet-getitem?mr=#1}{#2}
}
\providecommand{\href}[2]{#2}

\end{document}